\documentclass[12pt]{amsart}
\usepackage{amsthm}
\usepackage{amssymb}
\usepackage{amsmath}
\usepackage{hyperref}
\usepackage{xcolor}
\usepackage{amsthm}
\usepackage[matrix,arrow]{xy}
\usepackage{enumerate}

\usepackage[style=trad-abbrv,maxnames=99,maxalphanames=9, isbn=false, giveninits=true, doi=false, url=false]{biblatex}
\bibliography{ref.bib}
\renewbibmacro{in:}{}

\DeclareMathAlphabet\mathbfcal{OMS}{cmsy}{b}{n}

\numberwithin{equation}{section}

\newtheorem{theorem}{Theorem}[section]
\newtheorem{definition}[theorem]{Definition}

\newtheorem{lemma}[theorem]{Lemma}
\newtheorem{conjecture}[theorem]{Conjecture}
\newtheorem{proposition}[theorem]{Proposition}
\newtheorem{corollary}[theorem]{Corollary}
\newtheorem{remark}[theorem]{Remark}

\newtheoremstyle{named}{}{}{\itshape}{}{\bfseries}{.}{.5em}{\thmnote{#3\ }#1}
\theoremstyle{named}

 \newcommand{\RR}{\mathbb{R}}
 
 \newcommand{\NN}{{\mathbb N}}

\newcommand{\cE}{\mathcal{E}}

\newcommand{\cH}{\mathcal{H}}

\newcommand{\cJ}{\mathcal{J}}

\newcommand{\Ric}{\operatorname{Ric}}
\newcommand{\HRic}{\mathrm{HRic}}

\newcommand{\ddc}{\sqrt{-1}\partial\bar\partial}

\newcommand{\tr}{\mathrm{tr}}

\title{The Ricci iteration towards cscK metrics
}

\usepackage{hyperref}
\hypersetup{colorlinks,linkcolor={blue},citecolor={red}}

\begin{document}

\author[K. Zhang]{Kewei Zhang
}
\address{School of Mathematical Sciences, Beijing Normal University, Beijing, 100875, People's Republic of China.}
\email{kwzhang@bnu.edu.cn}

\dedicatory{Dedicated to Gang Tian on the occasion of his $65^{th}$ birthday}

\maketitle

\begin{abstract}
    Motivated by the problem of finding constant scalar curvature K\"ahler metrics, we investigate a Ricci iteration sequence of Rubinstein that discretizes the pseudo-Calabi flow. While the long time existence of the flow is still an open question, we show that the iteration sequence does exist for all steps, along which the K-energy decreases. We further show that the iteration sequence, modulo automorphisms, converges smoothly to a constant scalar curvature K\"ahler metric if there is one, thus confirming a conjecture of Rubinstein from 2007 and extending results of Darvas--Rubinstein to arbitrary K\"ahler classes.
\end{abstract}

\tableofcontents

\section{Introduction}

A long standing problem in K\"ahler geometry is to find constant scalar curvature K\"ahler (cscK) metrics in a given K\"ahler class. Namely, for a compact K\"ahler manifold $(X,\omega)$ of dimension $n$, we want to search for a K\"ahler form $\omega^*\in\{\omega\}$ that satisfies
$$
\tr_{\omega^*}\Ric(\omega^*)=\bar R,
$$
where $\bar R:=2\pi n\frac{-K_X\cdot\{\omega\}^{n-1}}{\{\omega\}^n}$ is the average of the scalar curvature. 

Regarding the existence of such metrics, the influential Tian's properness conjecture (cf. \cite[Remark 5.2]{T94},\cite[Conjecture 7.12]{T00notes}) predicts that the existence of cscK metrics is equivalent to some suitable notion of properness of Mabuchi’s K-energy functional.
Tian’s conjecture is central in K\"ahler geometry and has attracted much work over the past two decades including motivating much work on equivalence between algebro-geometric notions of stability and existence of canonical metrics, as well as on the interface of pluripotential theory and Monge--Amp\`ere equations. We refer to the surveys \cite{Thom06,PS09,T12,PSS12,Y20}. 

In \cite{DR17}, using the Finsler geometry of the space of K\"ahler metrics, the authors reduce Tian's conjecture to a purely PDE regularity problem, which has been recently solved in \cite{CC2}. Therefore we now have a complete solution to Tian's properness conjecture.

On the other hand, provided the properness of the K-energy, how to produce a cscK metric is also a challenging problem in its own right. In \cite{CC2} the authors show that certain continuity path provides an approach towards cscK metrics. In this work we show that one can also produce a cscK metric using some dynamical system.

To present our results, we begin by recalling an elementary result in K\"ahler geometry.
\begin{lemma}
    A closed $(1,1)$-form $\theta$ on $(X,\omega)$ satisfies $\tr_\omega\theta=\mathrm{Const.}$ if and only if $\theta$ is harmonic with respect to the K\"ahler metric $\omega$.
\end{lemma}

Therefore, $\omega^*$ is cscK if and only if $\Ric(\omega^*)$ is a harmonic form with respect to $\omega^*$. This viewpoint is by no means new, which was explored in early works of Calabi, Futaki, Bando and Mabuchi; see, e.g., \cite{Ban06}. When combined with the framework of geometric flows, this motivates one to
consider the following variant of the K\"ahler Ricci flow:
\begin{equation}
    \label{eqn:hKRF}
    \partial_t\omega_t=-\Ric(\omega_t)+\HRic(\omega_t),\ \omega_0=\omega.
\end{equation}
Here, given any K\"ahler form $\alpha$, $\HRic(\alpha)$ denotes the harmonic part of $\Ric(\alpha)$ with respect to $\alpha$.
If the flow \eqref{eqn:hKRF} smoothly converges to a limit $\omega_\infty$, then one has $$\Ric(\omega_\infty)=\HRic(\omega_\infty),$$ namely, $\omega_\infty$ is a cscK metric.

\begin{remark}
   If $2\pi c_1(X)=\lambda\{\omega\}$, then the flow \eqref{eqn:hKRF} reduces to
   $$
   \partial_t\omega_t=-\Ric(\omega_t)+\lambda\omega_t,\
   $$
   which is exactly the classical normalized K\"ahler Ricci flow in the study of K\"ahler--Einstein and K\"ahler Ricci soliton metrics (see e.g. \cite{Cao85,TZ07,PSSW09,TZ13,TZZZ13,TZ16,B18,CW20}). By \cite{Cao85,TZ06} we know that such flow has long time existence.
\end{remark}

The flow \eqref{eqn:hKRF} first appeared in \cite{Guan07} (see also \cite{Sim05} for a related flow) and was then briefly studied in \cite{Rub07,Rub08}. Later in \cite{CZ13}, this flow was systematically investigated and the authors call it the \emph{pseudo-Calabi flow}. Note that the flow \eqref{eqn:hKRF} can be viewed as a modified version of the Calabi flow \cite{Calabi82}, and it can be reduced to the coupled system of equations:
\begin{equation*}
    \label{eqn:parabo-CC-eqn}
    \begin{cases}
        \omega_t^n=e^{F_t}\omega^n,\\
        \dot{F_t}=\Delta_t F_t+\bar R-\tr_{\omega_t}\Ric(\omega).
    \end{cases}
\end{equation*}
So we obtain a parabolic version of the coupled equations for cscK metrics that are studied in \cite{CC1,CC2}. This being said, it is still a highly non-trivial problem to study this flow, with its long time existence and limiting behavior largely open. 

In this paper we adopt a somewhat different viewpoint.
We consider the discretization of the pseudo-Calabi flow \eqref{eqn:hKRF} that was first proposed by Rubinstein \cite{Rub07,Rub08}. More precisely,
given $\tau>0$, we investigate the following \emph{Ricci iteration} that appeared in \cite[Definition 2.1]{Rub07} and \cite[(41)]{Rub08}:

\begin{equation}
\label{eqn:def-har-Ric-ite}
    \frac{\omega_{i+1}-\omega_i}{\tau}=-\Ric(\omega_{i+1})+\HRic(\omega_{i+1}),\ i\in\NN,\ \omega_{0}=\omega.
\end{equation}

Part of the interest in this Ricci iteration is that, clearly, cscK metrics are fixed points. Therefore \eqref{eqn:def-har-Ric-ite} aims to provide a natural theoretical and numerical approach to uniformization 
in the challenging case of cscK metrics. In \cite[Conjecture 2.1]{Rub07}, Rubinstein proposed the following.

\begin{conjecture}
\label{conj:Rub}
    Let $X$ be a compact K\"ahler manifold, and assume that there exists a constant scalar curvature K\"ahler metric in a K\"ahler class $\Omega$. Then for any $\omega\in\Omega$ the Ricci iteration \eqref{eqn:def-har-Ric-ite} exists for all $i\in\NN$ and converges in an appropriate sense to a constant scalar curvature metric.
\end{conjecture}

In addition, the Ricci iteration could be a source of new insights for the study of the pseudo-Calabi flow, which is known to be a rather difficult problem in the field of geometric flows.
For instance, just as in the case of Calabi flow, the long time existence of the flow \eqref{eqn:hKRF} is still unknown (see \cite[Conjecture 8.3]{CZ13}). However, after discretization, we can prove the following long time existence result for the sequence \eqref{eqn:def-har-Ric-ite}.

\begin{theorem}
\label{thm:ite-exist}
    There exists a uniform constant $\tau_0\in(0,\infty]$, depending only on $X$ and the K\"ahler class $\{\omega\}$, such that for any $\tau\in(0,\tau_0)$ the iteration sequence \eqref{eqn:def-har-Ric-ite} exists for all $i\in\NN$, with each $\omega_i$ being uniquely determined by $\omega_0$.
\end{theorem}

This result gives a strong evidence for the long time existence of the pseudo-Calabi flow. Indeed, sending $\tau\to 0$, the iteration sequence $\{\omega_i\}$ is expected to converge to the flow \eqref{eqn:hKRF} (this is interestingly still a conjecture even for the Ricci iteration associated to the K\"ahler Ricci flow; compare the classical Rothe's method for parabolic equations). The fact that the sequence $\{\omega_i\}$ exists for all $i$ should imply that the flow \eqref{eqn:hKRF} exists for all $t$. This is of course a heuristic viewpoint, which hopefully can be made more rigorous in future study.

It is proved in \cite[Theorem 3.1]{CZ13} that Mabuchi's K-energy decreases along the flow \eqref{eqn:hKRF}.
We show that this is also the case for the Ricci iteration \eqref{eqn:def-har-Ric-ite}, which was previously only known in the case where $c_1(X)=\lambda\{\omega\}$ (see \cite[Proposition 4.2]{Rub08}).

\begin{theorem}
\label{thm:ite-decrease-K-energy}
    Along the iteration sequence $\{\omega_i\}_{i\in\NN}$, the K-energy $K_\omega$ satisfies
    $$
    K_\omega(\omega_{i+1})\leq K_\omega(\omega_i)\text{ for all }i\in\NN.
    $$
    The equality holds for some $i\in\NN$ if and only if $\omega_i=\omega_0$ is cscK for all $i\in\NN$. 
\end{theorem}

Since a cscK metric, if exists, minimizes the K-energy. The above result suggests that the iteration sequence \eqref{eqn:def-har-Ric-ite} has the tendency to be attracted by a cscK metric in a suitable sense. We show that this is indeed the case, thus confirming Conjecture \ref{conj:Rub}.

\begin{theorem}
\label{thm:smooth-convergence}
    Assume that there exists a cscK metric in $\{\omega\}$. Then for any $\tau>0$ the iteration sequence $\omega_i$ exists, and, up to biholomorphic automorphisms, converges to a cscK metric smoothly.
\end{theorem}

Our results extend those in the previous works \cite{Rub07,Rub08,Kel,BBEGZ19,DR19,Ryan}, where $c_1(X)$ is assumed to be proportional to $\{\omega\}$. Moreover, Theorem \ref{thm:smooth-convergence} also gives strong evidence that the pseudo-Calabi flow shall converge to a cscK metric, if there is one in $\{\omega\}$ (cf. \cite[Conjecture 7.4]{Rub08} and \cite[Question 8.5]{CZ13}). 
In view of Tian's properness conjecture, Theorem \ref{thm:smooth-convergence} also shows that the properness of K-energy (modulo group actions, in the sense of Definition \ref{def:K-proper-modulo-G}) implies 
that one can find a cscK metrics using the dynamical system \eqref{eqn:def-har-Ric-ite}.

Compared to the recent work of Darvas--Rubinstein \cite{DR19} in the Fano case, the main difficulty 
we are faced with is the lack of Ding functional in our general setting. As we shall see, this technicality can be circumvented with the help of the estimates in \cite[\S 3]{CC2}, which are needed for the smooth convergence in Theorem \ref{thm:smooth-convergence}.

For the direction of Ricci iteration in the real case, we refer the reader to \cite{PR19,Ryan,BPR21}. See also \cite{GKY13} for a Ricci iteration in the local setting.

\textbf{Organization.} After recalling some standard notions and facts in \S \ref{sec:functionals} and 
\S \ref{sec:d1}, we prove Theorem \ref{thm:ite-exist} and Theorem \ref{thm:ite-decrease-K-energy} in \S \ref{sec:Ric-ite}. Relying on \cite[\S 3]{CC2}, we will derive some a priori estimates for the Ricci iteration in \S \ref{sec:estimates}, which allows us to prove Theorem \ref{thm:smooth-convergence} in \S \ref{sec:smooth-conv}. Finally in \S \ref{sec:thR-ite} we point out that our work can be extended to the setting of twisted cscK metrics.

\newpage

\textbf{Acknowledgements.} It is the author's great pleasure to dedicate this paper to Prof. Tian, on the occasion of his 65th birthday, for his constant guidance over the years. The author is also grateful to T. Darvas, W. Jian, Y. Rubinstein and Y. Shi for helpful discussions and comments. Part of this work was done during the pleasant and inspiring visit at the Tianyuan Mathematics Research center in Oct. 2023. The author is also grateful to Y. Rubinstein for introducing this problem when the author was visiting the University of Maryland during 2017-2018. 

The author is supported by NSFC grants 12271040, and 12271038.

\section{Energy functionals}
\label{sec:functionals}
We recall several standard functionals that will be used throughout this paper.

Let $(X,\omega)$ be a compact K\"ahler manifold of dimension $n$, and set 
$$
\cH_\omega:=\{\varphi\in C^\infty(X,\RR)|\omega_\varphi:=\omega+\ddc \varphi>0\}.
$$
Put $V:=\int_X\omega^n.$ And let
$$
\Ric(\omega_\varphi):=-\ddc\log\det\omega_\varphi\in 2\pi c_1(X),\ R(\omega_\varphi):=\tr_{\omega_\varphi}\Ric(\omega_\varphi),
$$
$$
\bar R:=\frac{1}{V}\int_XR(\omega)\omega^n=\frac{n}{V}\int_X\Ric(\omega)\wedge\omega^{n-1}=2\pi n\frac{c_1(X)\cdot\{\omega\}^{n-1}}{\{\omega\}^n}.
$$

For any $u,v\in\cH_\omega$, define
$$
I(u,v)=I(\omega_u,\omega_v):=\frac{1}{V}\int_X(v-u)(\omega^n_u-\omega^n_v).
$$

$$
E(u,v):=\frac{1}{(n+1)V}\int_X(v-u)\sum_{i=0}^n\omega^i_u\wedge\omega^{n-i}_v.
$$

$$
J(u,v)=J(\omega_u,\omega_v):=\frac{1}{V}\int_X(v-u)\omega^n_u-E(u,v).
$$

$$
Ent(u,v)=Ent(\omega_u,\omega_v):=\frac{1}{V}\int_X\log\frac{\omega^n_v}{\omega^n_u}\omega^n_v.
$$
Note that by Jensen's inequality, it always holds that $Ent(u,v)\geq 0$.

For any closed $(1,1)$ form $\chi$, define
$$
\cJ^\chi(u,v):=\frac{1}{V}\int_X (v-u)\chi\wedge\sum_{i=0}^{n-1}\omega_u^i\wedge\omega^{n-1-i}_v-\bar\chi E(u,v),
$$
where
$$
\bar \chi:=\frac{n}{V}\int_X\chi\wedge\omega^{n-1}=n\frac{\{\chi\}\cdot\{\omega\}^{n-1}}{\{\omega\}^n}.
$$
The K-energy is defined by
$$
K(u,v)=K(\omega_u,\omega_v):=Ent(u,v)+\cJ^{-\Ric(\omega_u)}(u,v).
$$
The $\chi$-twisted K-energy is
$$
K^\chi(u,v)=K^\chi(\omega_u,\omega_v):=K(u,v)+\cJ^\chi(u,v).
$$

If we choose $\chi:=\omega_u$, then integration by parts gives
\begin{equation}
    \label{eq:cJuv=I-Juv}
    \cJ^{\omega_u}(u,v)=(I-J)(u,v).
\end{equation}
More generally, if $\chi:=\omega_w$ for some $w\in\cH_\omega$, then
$$
\cJ^{\omega_w}(u,v)=(I-J)(u,v)+\frac{1}{V}\int_X(w-u)(\omega^n_v-\omega^n_u),
$$
so in particular we have that
\begin{equation}
    \label{eq:cJuw=-Juw}
    \cJ^{\omega_w}(u,w)=-J(u,w).
\end{equation}

One has the following variation formulas (for any $u,v\in\cH_\omega$ and $f\in C^\infty(X,\RR)$):
\begin{equation}
\label{eq:var-formula}
    \begin{cases}
        \frac{d}{dt}\big|_{t=0}E(u,v+tf)=\frac{1}{V}\int_Xf\omega^n_v,\\
        \frac{d}{dt}\big|_{t=0}\cJ^\chi(u,v+tf)=\frac{1}{V}\int_Xf(\tr_{\omega_v}\chi-\bar\chi)\omega^n_v,\\
        \frac{d}{dt}\big|_{t=0}K^\chi(u,v+tf)=\frac{1}{V}\int_Xf(\bar R-\bar \chi+\tr_{\omega_v}\chi-R(\omega_v))\omega^n_v.\\
    \end{cases}
\end{equation}
They imply the well known cocycle relations (for $u,v,w\in\cH_\omega$):
$$
E(u,v)+E(v,w)=E(u,w).
$$
$$
\cJ^\chi(u,v)+\cJ^\chi(v,w)=\cJ^\chi(u,w).
$$
$$
K^\chi(u,v)+K^\chi(v,w)=K^\chi(u,w).
$$
One can then further deduce the following cocycle relations:
$$
J(u,v)+J(v,w)=J(u,w)+\frac{1}{V}\int_X(v-w)(\omega^n_u-\omega^n_v).
$$
$$
(I-J)(u,v)+(I-J)(v,w)=(I-J)(u,w)+\frac{1}{V}\int_X(v-u)(\omega^n_w-\omega^n_v).
$$

The following result proved in Tian's work \cite{Tian87} will be used repeatedly.
\begin{lemma}
\label{lem:J<I-J<J}
For any $u,v\in\cH_\omega$, it holds that
$$
\frac{1}{n}J(u,v)\leq(I-J)(u,v)\leq nJ(u,v).
$$
Moreover, one has
$$
I(u,v)\geq0,\ J(u,v)\geq0,\ (I-J)(u,v)\geq0.
$$
If one of the inequalities is an equality, then they all are, in which case $u=v$.
\end{lemma}

\begin{lemma}
\label{lem:cJ-cJ=J}
    For any $u,v,w\in\cH_\omega$ one has
    $$
    \cJ^{\omega_w}(u,v)-\cJ^{\omega_w}(u,w)=J(v,w).
    $$
    So in particular, $\cJ^{\omega_w}(u,v)\geq\cJ^{\omega_w}(u,w)$, and the equality holds if and only if $v=w$.
\end{lemma}
\begin{proof}
    Using cocycle relation, we can write (recall \eqref{eq:cJuw=-Juw})
    \begin{equation*}
        \begin{aligned}
            \cJ^{\omega_w}(u,v)-\cJ^{\omega_w}(u,w)=-\cJ^{\omega_w}(v,w)=J(v,w)\\
        \end{aligned}
    \end{equation*}
  to conclude.
\end{proof}

\textbf{Convention.}
Given an energy functional $F\in\{I,J,\cJ^\chi,K,K^\chi\}$ and $u\in\cH_\omega$, we also use the notation
$$
F_{\omega}(u)=F_{\omega}(\omega_u):=F(0,u)\text{ and }E_{\omega}(u):=E(0,u)
$$
in the circumstances where $\omega$ is viewed as a background metric.

\begin{definition}
    The twisted K-energy $K^\chi_\omega$ is said to be proper if there exist $\gamma>0$ and $C>0$ such that
    $$
    K^\chi_\omega(u)\geq\gamma(I_\omega-J_\omega)(u)-C\text{ for all }u\in\cH_\omega.
    $$
\end{definition}

\section{The metric completion}
\label{sec:d1}
We will work with the finite energy space $\cE_\omega^1$ introduced in \cite{GZ07} and use the $d_1$-distance on it introduced by Darvas \cite{Dar15}. They provide useful tools for proving our main result concerning convergence of the Ricci iteration. We briefly recall the machinery, referring to \cite{DarvasSurvey} and references therein for more details.

Let
$$
\mathrm{PSH}_\omega=\{\varphi\in L^1(\omega^n)|\varphi:X\to[-\infty,\infty)\\
\text{ upper semi-continuous, }\omega_\varphi=\omega+\ddc\varphi\geq 0\}.
$$
Following Guedj--Zeriahi \cite[Definition 1.1]{GZ07} we define the subset of full mass potentials:
$$
\cE_\omega:=\big\{\varphi\in\mathrm{PSH}_\omega:\lim_{j\to-\infty}\int_{\{\varphi\leq j\}}(\omega+\ddc\max\{\varphi,j\})^n=0\big\}.
$$
For each $\varphi\in\cE_\omega$, define $\omega^n_\varphi:=\lim_{j\to-\infty}\textbf{1}_{\{\varphi>j\}}(\omega+\ddc\max\{\varphi,j\})^n$. The measure $(\omega+\ddc\max\{\varphi,j\})^n$ is defined by the work of Beford--Taylor \cite{BT82} since $\max\{\varphi,j\}$ is bounded. Consequently, $\varphi\in\cE_\omega$ if and only if $\int_X\omega^n_\varphi=\int_X\omega^n$, justifying the name of $\cE_\omega$.

Next, define a further subset, the space of finite $1$-energy potentials:
$$
\cE_\omega^1:=\big\{\varphi\in\cE_\omega:\int_X|\varphi|\omega^n_\varphi<\infty\big\}.
$$
Consider the following weak Finsler metric on $\cH_\omega$ \cite{Dar15}:
$$
||\xi||_\varphi:=\frac{1}{V}\int_X|\xi|\omega^n_\varphi,\ \xi\in T_\varphi\cH_\omega=C^\infty(X,\RR).
$$
We denote by $d_1$ the associated pseudo-metric and recall the following result characterizing the $d_1$-metric completion of $\cH_\omega$ \cite[Theorem 3.5]{Dar15}:

\begin{theorem}
    $(\cH_\omega,d_1)$ is a metric space whose completion can be identified with $(\cE_\omega^1)$, where
    $$
    d_1(u_0,u_1):=\lim_{k\to\infty}d_1(u_0^k,u_1^k),
    $$
    for any smooth decreasing sequences $\{u_i^k\}_{k\in\NN}\subset\cH_\omega$ converging pointwise to $u_i\in\cE^1_\omega$, $i=0,1$.
\end{theorem}

Let us now recall several properties of the $d_1$ metric that will be used in this paper.

\begin{lemma}(\cite[Theorem 3]{Dar15})
\label{lem:d1-comparable-L1}
    For any $u,v\in\cE^1_\omega$, one has
    $$
    \frac{1}{C}d_1(u,v)\leq\int_X|u-v|(\omega^n_u+\omega^n_v)\leq C d_1(u,v)
    $$
    for some dimensional constant $C>1$.
\end{lemma}

\begin{lemma}(\cite[Proposition 5.5]{DR17})
\label{lem:J-d-1-compare}
    There exists a constant $C>1$ depending only on $(X,\omega)$ such that
    $$
    \frac{1}{C}J_\omega(\varphi)-C\leq d_1(0,\varphi)\leq CJ_\omega(\varphi)+C\text{ for any }\varphi\in\cH_0.
    $$
\end{lemma}

One can extend the functionals $Ent_\omega, I_\omega, J_\omega, E_\omega$ and $\cJ^\chi_\omega$ to the space $\cE^1_\omega$.

\begin{lemma}(\cite{BDL17})
\label{lem:funtionals-to-E1}
    All the functionals $I_\omega,J_\omega,E_\omega,\cJ_\omega^\chi$ are $d_1$-continuous. The entropy $Ent_\omega$ is $d_1$-lower semi-continuous ($d_1$-lsc). Moreover, for any $u\in\cE^1_\omega$, there exists $\cH_\omega\ni u_i\xrightarrow{d_1}u$ such that $Ent_\omega(u_i)\to Ent_\omega(u)$.
\end{lemma}

We need the following compactness result going back to \cite{BBEGZ19} (see \cite[Theorem 5.6]{DR17} for a convenient formulation for our context).
\begin{lemma}
\label{lem:d1-ent-bound-cpt}
    For any $A>0$, the set
    $$
    \{\varphi\in\cE^1_\omega|d_1(0,\varphi)\leq A\text{ and }Ent_\omega(\varphi)\leq A\}
    $$
    is compact in $\cE^1_\omega$ with respect to the $d_1$-topology.
\end{lemma}

For any $u,v\in\cE^1_\omega$, one can also define
$$
I(u,v):=\frac{1}{V}\int_X(v-u)(\omega^n_u-\omega^n_v),
$$
$$
E(u,v):=E_\omega(v)-E_\omega(u),
$$
and
$$
J(u,v):=\frac{1}{V}\int_X(v-u)\omega^n_u-E(u,v).
$$
Recall that (see \cite[Proposition 3.40]{DarvasSurvey})
\begin{equation}
    \label{eq:Euv<d1-uv}
    |E(u,v)|\leq d_1(u,v).
\end{equation}

\begin{lemma}
\label{lem:J-d1-conti}
    Given $u_i,u,v_i,v\in \cE^1_\omega$ such that $u_i\xrightarrow{d_1}u$ and $v_i\xrightarrow{d_1}v$, then
    $$
    \lim_{i\to\infty}I(u_i,v_i)=I(u,v),\ 
    \lim_{i\to\infty}J(u_i,v_i)=J(u,v).
    $$
\end{lemma}

\begin{proof}
We only deal with the $J$-functional, since the proof for the $I$-functional is similar.
    One has
    \begin{equation*}
        \begin{aligned}
             |J(u_i,v_i)-J(u,v)|&\leq|\frac{1}{V}\int_X(v_i-u_i)\omega^n_{u_i}-\frac{1}{V}\int_X(v-u)\omega^n_u|+|E(u,u_i)|+|E(v,v_i)|.\\
        \end{aligned}
    \end{equation*}
    By \eqref{eq:Euv<d1-uv} it suffices to estimate $|\int_X(v_i-u_i)\omega^n_{u_i}-\int_X(v-u)\omega^n_u|$, which can be bounded from above by
    \begin{equation*}
    \begin{aligned}
         |\int_X(v_i-v)&(\omega^n_{u_i}-\omega^n_u)|+|\int_X(u_i-u)(\omega^n_{u_i}-\omega^n_u)|+\int_X(|v_i-v|+|u_i-u|)\omega^n_u\\
         &+|\int_Xv(\omega^n_{u_i}-\omega^n_u)|+|\int_Xu(\omega^n_{u_i}-\omega^n_u)|.
    \end{aligned}
    \end{equation*}
    All of these terms go to zero, thanks to \cite[Proposition 3.48 and Corollary 3.51]{DarvasSurvey} (see also \cite[Lemma 3.13, Lemma 5.8]{BBGZ}).
\end{proof}

The next quasi-triangle inquality is proved in \cite[Theorem 1.8]{BBEGZ19}.
\begin{lemma}
\label{lem:triangle-ineq-for-I}
    There exists a dimensional constant $C_n>0$ such that for any $u,v,w\in\cE^1_\omega$
    $$
    I(u,v)\leq C_n(I(u,w)+I(w,v)).
    $$
\end{lemma}

We will need the following convergence criterion, which is a simple consequence of \cite[Proposition 2.3]{BBEGZ19}.
\begin{lemma}
    Assume that $u_i, u\in\cE^1_\omega$ such that
    $E_\omega(u_i)=E_\omega(u)=0$ and $d_1(0,u_i)\leq A$ for some $A>0$ independent of $i$. Then
    $$
    u_i\xrightarrow{d_1}u\Leftrightarrow I(u_i,u)\to 0.
    $$
\end{lemma}

\begin{proof}
That $u_i\xrightarrow{d_1}u$ implies $I(u_i,u)\to 0$ follows from Lemma \ref{lem:d1-comparable-L1}.

Now assume that $I(u_i,u)\to 0$.
    By Lemma \ref{lem:d1-comparable-L1} the bound on $d_1(0,u_i)$ implies the $L^1$ bound for $u_i,$ which in turn implies a bound on $|\sup_X u_i|$. Put $u_i':=u_i-\sup_X u_i$ and $u':=u-\sup_X u$. Then $I(u_i',u')\to 0$ as well. By \cite[Proposition 2.3]{BBEGZ19} we know that $u_i'\xrightarrow{d_1}u'$ and hence $E_\omega(u_i')=-\sup_X u_i\to E_\omega(u')=-\sup_X u$. Then from $\sup_X u_i\to\sup_X u$ and $u_i'\xrightarrow{d_1}u'$ we deduce that $u_i\xrightarrow{d_1}u$ (using Lemma \ref{lem:d1-comparable-L1}). So we conclude.
\end{proof}

\begin{corollary}
\label{cor:ui-vi-same-d1-lim}
    Given two sequences $u_i, v_i\in\cE^1_\omega$ with $E_\omega(u_i)=E_\omega(v_i)=0$. Assume that $u_i\xrightarrow{d_1}u$ and $I(u_i,v_i)\to 0$. Then $v_i\xrightarrow{d_1}u$ as well.
\end{corollary}

\begin{proof}
    First, that $u_i\xrightarrow{d_1}u$ implies that $I(0,u_i)\leq C$ by Lemma \ref{lem:d1-comparable-L1}. So one has
    $$
    J(0,v_i)\leq I(0,v_i)\leq C_n(I(0,u_i)+I(u_i,v_i))\leq C'.
    $$
    This implies that $d_1(0,v_i)\leq C''$ by Lemma \ref{lem:J-d-1-compare}. Moreover, Lemma \ref{lem:triangle-ineq-for-I} and \ref{lem:d1-comparable-L1} imply that
    $$
    I(v_i,u)\leq C_n(I(v_i,u_i)+I(u_i,u))\to 0.
    $$
    So we conclude from the previous lemma that $v_i\xrightarrow{d_1}u$.
\end{proof}

We will frequently use the space
$$
\cH_0:=\{\varphi\in\cH_\omega|E_\omega(\varphi)=0\}.
$$

Recall that the action of
$$G:=Aut_0(X),$$
the connected component of complex Lie group of holomorphic automorphisms of $X$,
preserves the K\"ahler class, so $G$ naturally acts on $\cH_0$ in the following way.

\begin{lemma}(\cite[Lemma 5.8]{DR17})
    For any $\varphi\in\cH_0$ and $f\in G$, let $f.\varphi\in\cH_0$ be the unique potential such that $f^*\omega_\varphi=\omega_{f.\varphi}$. Then
    $$
    f.\varphi=f.0+f^*\varphi.
    $$
\end{lemma}

\begin{lemma}
\label{lem:I&J-G-inv}
    For any $u,v\in\cH_0$ and $f\in G$, one has
    $$
    I(u,v)=I(f.u,f.v),\ J(u,v)=J(f.u,f.v).
    $$
\end{lemma}

\begin{proof}
    One has
    \begin{equation*}
        \begin{aligned}
            I(u,v)&=\frac{1}{V}\int_X(v-u)(\omega^n_u-\omega^n_v)\\
            &=\frac{1}{V}\int_X(f^*v-f^*u)(f^*\omega^n_u-f^*\omega^n_v)\\
            &=I(f^*\omega_u,f^*\omega_v)=I(\omega_{f.u},\omega_{f.v})=I(f.u,f.v).\\
        \end{aligned}
    \end{equation*}
    For $J$-functional, we can write
    \begin{equation*}
        \begin{aligned}
            J(u,v)&=\frac{1}{V}\int_X(v-u)\omega^n_u=\frac{1}{V}\int_X(f^*v-f^*u)f^*\omega^n_u\\
            &=\frac{1}{V}(f^*v+f.0-f^*u-f.0)\omega^n_{f.u}\\
            &=\frac{1}{V}\int_X(f.v-f.u)\omega^n_{f.u}=J(f.u,f.v).\\
        \end{aligned}
    \end{equation*}
\end{proof} 

Finally, we recall that
$G$ acts on $\cH_0$ isometrically.
\begin{lemma}(\cite[Lemma 5.9]{DR17})
    For any $u,v\in\cH_0$ and $f\in G$ one has
    $$
    d_1(u,v)=d_1(f.u,f.v).
    $$
\end{lemma}

Then define (as in \cite{DR17})
\begin{equation}
    \label{eq:def-d-G}
    d_{1,G}(u,v):=\inf_{f\in G}d_1(u,f.v).
\end{equation}

\begin{definition}
\label{def:K-proper-modulo-G}
    The K-energy $K_\omega$ is said to be proper modulo $G$ if there exist $\gamma>0$ and $C>0$ such that
    $$
    K_\omega(u)\geq\gamma d_{1,G}(0,u)-C\text{ for all }u\in\cH_0.
    $$
\end{definition}

\section{Basic properties of the Ricci iteration}
\label{sec:Ric-ite}

In this part, we prove Theorem \ref{thm:ite-exist} and Theorem \ref{thm:ite-decrease-K-energy}.

We begin by introducing the following analytic threshold.
\begin{equation}
    \label{eq:def-gamma}
    \gamma(X,\{\omega\}):=\sup\bigg\{\gamma\in\RR\bigg|\inf_{\varphi\in\cH_\omega}(K_{\omega}(\varphi)-\gamma(I_\omega-J_\omega)(\varphi))>-\infty\bigg\},
\end{equation}

\begin{lemma}
\label{lem:gamma-indpd-of-metric}
    The threshold $\gamma(X,\{\omega\})$ satisfies $\gamma(X,\{\omega\})>-\infty$ and is independent of the choice of
$\omega$ in its cohomology class, hence the notation.
\end{lemma}

\begin{proof}
That $\gamma(X,\{\omega\})>-\infty$ follows from
$$
K_\omega(\varphi)\geq\cJ_\omega^{-\Ric(\omega)}(\varphi)\geq-Cd_1(0,\varphi),\ \varphi\in\cH_0,
$$
where we used \cite[(4.2)]{CC2}. So, by Lemma \ref{lem:J-d-1-compare} and Lemma \ref{lem:J<I-J<J}, one can find $C\gg0$ such that $K_\omega(\varphi)+C(I_\omega-J_\omega)(\varphi)\geq-C$ for all $\varphi\in\cH_\omega$.

To show that $\gamma(X,\{\omega\})$ does not depend on the choice of $\omega$, we use the cocycle relations recalled in \S \ref{sec:functionals}. It suffices to note the estimate
$$
|(I-J)(u,w)-(I-J)(v,w)|\leq|(I-J)(u,v)|+\frac{1}{V}\int_X|u-v|(\omega^n_w+\omega^n_v)\leq C||u-v||_{C^0},
$$
for any $u,v,w\in\cH_\omega$, where $C>0$ is a dimensional constant.
\end{proof}

\begin{proof}[Proof of Theorem \ref{thm:ite-exist}]
Taking trace of \eqref{eqn:def-har-Ric-ite} we see that
\begin{equation}
    \label{eqn:tcscK-eq}
    R(\omega_{i+1})=\bar R+\frac{\tr_{\omega_{i+1}}\omega_i-n}{\tau},
\end{equation}
which is a twisted cscK equation. As we now argue, given any $\omega_i\in\{\omega\}$, the existence of $\omega_{i+1}\in\{\omega\}$ solving \eqref{eqn:tcscK-eq} is guaranteed by the main result in \cite{CC2} (see also \cite{Hash19}), once $\tau$ is chosen to be small enough.

Indeed, for any $\gamma<\gamma(X,\{\omega\})$ and any K\"ahler form $\alpha\in\{\omega\}$, the twisted K-energy $K_\alpha-\gamma(I_\alpha-J_\alpha)$ is proper by Lemma \ref{lem:gamma-indpd-of-metric}.
Now choosing $\tau_0\in(0,\infty]$ so that $-1/\tau_0\leq\gamma(X,\{\omega\})$, we see that for any $\tau\in(0,\tau_0)$ and $\omega_i\in\{\omega\}$, the $\frac{\omega_i}{\tau}$-twisted K-energy 
$$
K^{\frac{\omega_i}{\tau}}_{\omega_i}=K_{\omega_i}+\cJ^{\frac{\omega_i}{\tau}}_{\omega_i}=K_{\omega_i}+\frac{1}{\tau}(I_{\omega_i}-J_{\omega_i})
$$
is proper (here we used \eqref{eq:cJuv=I-Juv}), which implies the solvability of \eqref{eqn:tcscK-eq} by \cite[Theorem 4.1]{CC2}. Moreover, such $\omega_{i+1}$ is uniquely determined, by \cite[Theorem 4.13]{BDL17}. This completes the proof of Theorem \ref{thm:ite-exist}.
It is also clear that one can take $\tau_0=\infty$ once $\gamma(X,\{\omega\})\geq0$.

\end{proof}

\begin{proof}[Proof of Theorem \ref{thm:ite-decrease-K-energy}]
Notice that $\omega_{i+1}$ minimizes the twisted K-energy $K_{\omega_i}^{\frac{\omega_i}{\tau}}$ (see \cite[Corollary 4.5]{CC2}), so that
$$
K_{\omega_i}(\omega_{i+1})+\frac{1}{\tau}(I_{\omega_i}-J_{\omega_i})(\omega_{i+1})=K^{\frac{\omega_i}{\tau}}_{\omega_i}(\omega_{i+1})\leq K^{\frac{\omega_i}{\tau}}_{\omega_i}(\omega_i)=0.
$$
This implies that
$$
K_{\omega}(\omega_{i+1})-K_{\omega}(\omega_i)=K_{\omega_i}(\omega_{i+1})\leq-\frac{1}{\tau}(I_{\omega_i}-J_{\omega_i})(\omega_{i+1})\leq0,
$$
thanks to the cocycle property of the K-energy.

When the equality holds for some $i$, one has $(I_{\omega_i}-J_{\omega_i})(\omega_{i+1})=0$, which means that $\omega_i=\omega_{i+1}$. Then \eqref{eqn:tcscK-eq} shows that $\omega_i=\omega_{i+1}$ are both cscK metrics. Moreover, from
$$
R(\omega_i)=\bar R+\frac{\tr_{\omega_i}\omega_{i-1}-n}{\tau}
$$
we get $\tr_{\omega_i}\omega_{i-1}=n$, and hence $\omega_{i-1}$ is harmonic with respect to $\omega_i$. This forces that $\omega_{i-1}=\omega_i$, by the uniqueness of harmonic forms. So we eventually see that $\omega_0=\omega_i$ for all $i$, which is a fixed cscK metric. Thus we conclude Theorem \ref{thm:ite-decrease-K-energy}.
\end{proof}

\section{A priori estimates of the Ricci iteration}
\label{sec:estimates}

In this part we derive some a priori estimates for the iteration sequence \eqref{eqn:def-har-Ric-ite}. By Theorem \ref{thm:ite-exist} we can find some $\tau>0$ so that the iteration carries on forever. Up to scaling the K\"ahler class, we assume without loss of generality that $\tau=1$. Taking trace of \eqref{eqn:def-har-Ric-ite} we then have
$$
R(\omega_{i+1})=\bar R-n+\tr_{\omega_{i+1}}\omega_i,\ \omega_0=\omega.
$$
Write
$$
\omega_{i}=\omega+\ddc u_i, u_i\in\cH_0.
$$
Also, let $F_i\in C^\infty(X,\RR)$ be such that
\begin{equation}
    \label{eq:def-F-i-by-vol-form}
    (\omega+\ddc u_i)^n=e^{F_i}\omega^n.
\end{equation}
Then
\begin{equation}
    \label{eq:laplace-of-Fi}
    \Delta_{\omega_{i}}F_i=\tr_{\omega_i}(\Ric(\omega)-\omega_{i-1})+n-\bar R.
\end{equation}
In other words,
\begin{equation*}
    \label{eq:laplace-of-Fi'}
    \Delta_{\omega_{i}}(F_i+u_{i-1})=\tr_{\omega_i}(\Ric(\omega)-\omega)+n-\bar R.
\end{equation*}
Therefore, we are in the situation considered in \cite[\S 3]{CC2}.

We first derive the $C^0$ bound for $u_i$ and $F_i$.
\begin{proposition}
\label{prop:C0}
    Assume that there is some constant $A>0$ such that
    $$
    Ent_\omega(u_i)+d_1(0,u_i)\leq A\text{ for all }i\in\NN.
    $$
    Then there exists some constant $B_1$ depending only on $X,\omega$ and $A$ such that
    $$
    |F_i|+|u_i|\leq B_1\text{ for all }i\in\NN.
    $$
\end{proposition}

\begin{proof}
    First, using Lemma \ref{lem:d1-comparable-L1}, the bound $d_1(0,u_i)\leq A$ implies that the $u_i$ has uniform $L^1$ bound, which in turn gives that (see e.g. \cite[Lemma 3.45]{DarvasSurvey})
    $$
    |\sup_X u_i|\leq C_1\text{ for all }i\in\NN,
    $$
    where $C_1=C_1(X,\omega,A)>0$. Moreover, for any $p>1$, Zeriahi's version of the Skoda--Tian type estimate \cite{Zeri01} (see \cite[Corollary 4.16]{DarvasSurvey} for a formulation that fits our context) implies that there exists $C_2=C_2(X,\omega,A,p)>0$ such that
    \begin{equation}
        \label{eq:skoda-estimate-for-u-i}
         \int_Xe^{-p u_i}\omega^n\leq C_2\text{ for all }i\in\NN.
    \end{equation}
    Then one can apply \cite[Corollary 3.2]{CC2} (with $F=F_i$, $f_*=u_{i-1}$, $\varphi=u_i$, $\beta_0=\omega$ and $R_0=\bar R-n$) to find a constant $C_3=C_3(X,\omega,A)>0$ such that
    $$
    F_i+u_{i-1}\leq C_3,\ |u_i|\leq C_3\text{ for all }i\in\NN.
    $$
    Since $u_0=0$, we conclude by induction that there exists some $C_4=C_4(X,\omega,A)>0$ such that
    $$
     F_i\leq C_4,\ |u_i|\leq C_4\text{ for all }i\in\NN.
    $$
    Then \cite[Lemma 3.3]{CC2} further implies that there exists some $B_1=B_1(X,\omega,A)>0$ such that
    $$
     |F_i|+|u_i|\leq B_1\text{ for all }i\in\NN.
    $$
\end{proof}

\begin{corollary}
    Assume that there is some constant $A>0$ such that
    $$
    Ent_\omega(u_i)+d_1(0,u_i)\leq A\text{ for all }i\in\NN.
    $$
    Then for any $q>1$ there exists some constant $B_2\geq 1$ depending only on $X,\omega, q$ and $A$ such that
    $$
    n+\Delta_\omega u_i\geq\frac{1}{B_2}\text{ for all }i\in\NN.
    $$
    and
    \begin{equation}
        \label{eq:Lp-bound-for-n+Delta-ui}
         \int_X(n+\Delta_\omega u_i)^q\omega^n\leq B_2\text{ for all }i\in\NN.
    \end{equation}
\end{corollary}

\begin{proof}
The first inequality follows from
$$
n+\Delta_\omega u_i\geq n e^{F/n}\geq ne^{-B_1/n}.
$$
    The second inequality follows from \eqref{eq:skoda-estimate-for-u-i}, Proposition \ref{prop:C0} and \cite[Corollary 3.4]{CC2}.
\end{proof}

\begin{proposition}
    Assume that there is some constant $A>0$ such that
    $$
    Ent_\omega(u_i)+d_1(0,u_i)\leq A\text{ for all }i\in\NN.
    $$
    Then there exists some constant $B_3$ depending only on $X,\omega$ and $A$ such that
    $$
    \max_X|\nabla_{\omega_i}(F_i+u_{i-1})|^2_{\omega_i}+
  \max_X (n+\Delta_\omega u_i)\leq B_3\text{ for all }i\in\NN.
    $$
\end{proposition}

\begin{proof}
    The proof follows closely the one of \cite[Theorem 3.2]{CC2}. The basic idea is to estimate
    $$
    \Delta_{\omega_i}(e^{\frac{1}{2}(F_i+u_{i-1})}|\nabla_{\omega_i}(F_i+u_{i-1}))|^2_{\omega_i}+(n+\Delta_{\omega}u_i))
    $$
    and then apply Nash--Moser iteration. Compared to \cite{CC2}, the only difference is that we have the additional term $(n+\Delta_\omega u_i)$.

    To simplify the notation, we put $\Delta:=\Delta_\omega$, and use the subscript $i$ to denote the operators associated with the metric $\omega_i$, e.g., $\tr_i:=\tr_{\omega_i}$, $\Delta_i:=\Delta_{\omega_i}$. Also, put 
    $$
    w_i:=F_i+u_{i-1}.
    $$
    So one has
    $$
    \Delta_i w_i=\tr_i(\Ric(\omega)-\omega)+n-\bar R.
    $$
    
In what follows, the constants $C>0$ will change from line to line, which are all uniform (may depend on $X,\omega,A$, but are independent of $i$).

    Now we compute $\Delta_i(n+\Delta u_i)$. As in \cite[(3.25)]{CC2}, we have
    \begin{equation*}
        \begin{aligned}
             \Delta_i(n+\Delta u_i)&\geq -C\tr_i\omega(n+\Delta u_i)+\Delta F_i-R(\omega)\\
             &\geq -C\tr_i\omega(n+\Delta u_i)+\Delta w_i-(n+\Delta u_{i-1})-C.\\
        \end{aligned}
    \end{equation*}
Using \eqref{eq:def-F-i-by-vol-form} and Proposition \ref{prop:C0} one can estimate
$$
\tr_i\omega\leq n e^{-F_i}(n+\Delta u_i)^{n-1}\leq C(n+\Delta u_i)^{n-1}.
$$
Also, one can estimate $\Delta w_i$:
$$
|\Delta w_i|\leq \frac{1}{2C}\sum_k\frac{|(w_i)_{k\bar k}|^2}{(1+(u_i)_{k\bar k})^2}+\frac{C}{2}\sum_k(1+(u_i)_{k\bar k})^2\leq \frac{1}{2C}\sum_k\frac{|(w_i)_{k\bar k}|^2}{(1+(u_i)_{k\bar k})^2}+\frac{C}{2}(n+\Delta u_i)^2.
$$
So we get
\begin{equation}
    \label{eq:Delta-i-n+Delta-u-i}
    \Delta_i(n+\Delta u_i)\geq-C(n+\Delta u_i)^n-\frac{1}{2C}\sum_k\frac{|(w_i)_{k\bar k}|^2}{(1+(u_i)_{k\bar k})^2}-\frac{C}{2}(n+\Delta u_i)^2-(n+\Delta u_{i-1})-C.
\end{equation}

Next, we compute $\Delta_i(e^{\frac{1}{2}w_i}|\nabla_i w_i|^2_i)$. As in \cite[(3.43)-(3.49)]{CC2}, we have
\begin{equation*}
    \begin{aligned}
        \Delta_i(e^{\frac{1}{2}w_i}|\nabla_i w_i|^2_i)&\geq 2 e^{\frac{1}{2}w_i}\nabla_i w_i\cdot_i\nabla_i\Delta_i w_i+e^{\frac{1}{2}w_i}\sum_{k,l}\frac{|(w_i)_{k\bar l}|^2}{(1+(u_i)_{k\bar k})(1+(u_i)_{l\bar l})}\\
        &-C e^{\frac{1}{2}w_i}|\nabla_i w_i|^2\big(2e^{-F_i}(n+\Delta u_i)^{2n-1}+e^{-F_i}(n+\Delta u_i)^{n-1}+1\big).\\
        &\geq 2 e^{\frac{1}{2}w_i}\nabla_i w_i\cdot_i\nabla_i\Delta_i w_i+\frac{1}{C}\sum_{k,l}\frac{|(w_i)_{k\bar l}|^2}{(1+(u_i)_{k\bar k})(1+(u_i)_{l\bar l})}\\
        &-C e^{\frac{1}{2}w_i}|\nabla_i w_i|^2\big((n+\Delta u_i)^{2n-1}+(n+\Delta u_i)^{n-1}+1\big).
    \end{aligned}
\end{equation*}
The first inequality is just \cite[(3.49)]{CC2}, with one additional term that corresponds to the term $e^{\frac{1}{2}w}g_\varphi^{i\bar j}g_\varphi^{\alpha\bar\beta}w_{,\alpha\bar j}w_{,\bar \beta i}$ in \cite[(3.43)]{CC2}. This term is dropped in \cite[(3.49)]{CC2} since it plays no role in \emph{loc. cit.} But we need to keep it in order to dominate the bad term $|(w_i)_{k\bar k}|^2$ in \eqref{eq:Delta-i-n+Delta-u-i}. In the second inequality, we used Proposition \ref{prop:C0}. Here we also used the notation `$._i$' to denote the inner product taken with respect to the metric $\omega_i$.

Putting these estimates together, we then arrive at
\begin{equation*}
    \begin{aligned}
        \Delta_i(e^{\frac{1}{2}w_i}|\nabla_i w_i|^2_i+(n+&\Delta u_i))\geq 2 e^{\frac{1}{2}w_i}\nabla_i w_i\cdot_i\nabla_i\Delta_i w_i-C\big((n+\Delta u_i)^n+(n+\Delta u_i)^2+1\big)\\
        &-C e^{\frac{1}{2}w_i}|\nabla_i w_i|^2\big((n+\Delta u_i)^{2n-1}+(n+\Delta u_i)^{n-1}+1\big)-(n+\Delta u_{i-1}).\\
    \end{aligned}
\end{equation*}
Set
$$
U_i:=e^{\frac{1}{2}w_i}|\nabla_i w_i|^2_i+(n+\Delta u_i)+1.
$$
Using $n+\Delta u_i\geq B_2^{-1}$ and  $U_i\geq 1$ we can further simplify to get
\begin{equation*}
    \begin{aligned}
        \Delta_i U_i&\geq 2 e^{\frac{1}{2}w_i}\nabla_i w_i\cdot_i\nabla_i\Delta_i w_i-CU_i\big((n+\Delta u_i)^{2n-1}+1\big)-(n+\Delta u_{i-1})\\
        &\geq 2 e^{\frac{1}{2}w_i}\nabla_i w_i\cdot_i\nabla_i\Delta_i w_i-CU_i\big((n+\Delta u_i)^{2n-1}+(n+\Delta u_{i-1})+1\big).
    \end{aligned}
\end{equation*}
Put
$$
\tilde G_i:=C\big((n+\Delta u_i)^{2n-1}+(n+\Delta u_{i-1})+1\big),
$$
then we have the following key estimate:
$$
\Delta_i U_i\geq 2 e^{\frac{1}{2}w_i}\nabla_i w_i\cdot_i\nabla_i\Delta_i w_i-U_i\tilde G_i.
$$

Then for any $p>1$, we obtain
\begin{equation*}
    \begin{aligned}
        \int_X(p-1)&U_i^{p-2}|\nabla_i U_i|_i^2\omega_i^n=\int_XU_i^{p-1}(-\Delta_i U_i)\omega^n_i\\
        &\leq\int_XU^p_i\tilde G_i\omega^n_i-\int_X2U^{p-1}_i e^{\frac{1}{2}w_i}\nabla_i w_i\cdot_i\nabla_i\Delta_i w_i\omega^n_i.\\
    \end{aligned}
\end{equation*}
One can deal with the bad term $e^{\frac{1}{2}w_i}\nabla_i w_i\cdot_i\nabla_i\Delta_i w_i$ using integration by parts:
\begin{equation*}
    \begin{aligned}
       -\int_X2U^{p-1}_i&e^{\frac{1}{2}w_i}\nabla_i w_i\cdot_i\nabla_i\Delta_i w_i\omega^n_i=\int_X2U_i^{p-1}e^{\frac{1}{2}w_i}(\Delta_i w_i)^2\omega^n_i\\
       &+\int_XU^{p-1}_ie^{\frac{1}{2}w_i}|\nabla_i w_i|^2_i\Delta_i w_i\omega^n_i+\int_X2(p-1)U_i^{p-2}e^{\frac{1}{2}w_i}\nabla_i U_i\cdot_i\nabla_i w_i\Delta_i w_i\omega^n_i.\\
    \end{aligned}
\end{equation*}
Using the simple fact that $e^{\frac{1}{2}w_i}|\nabla_i w_i|^2_i\leq U_i$, we can estimate
$$
\int_XU^{p-1}_ie^{\frac{1}{2}w_i}|\nabla_i w_i|^2_i\Delta_i w_i\omega^n_i\leq \int_XU_i^p|\Delta_i w_i|\omega^n_i\leq \int_XU_i^p((\Delta_i w_i)^2+1)\omega^n_i,
$$
and by the Cauchy–Schwarz inequality together with the inequality of arithmetic and geometric means, we derive that
\begin{equation*}
    \begin{split}
        \int_X2U_i^{p-2}e^{\frac{1}{2}w_i}\nabla_i U_i\cdot_i\nabla_i w_i\Delta_i w_i\omega^n_i&\leq \int_X\frac{1}{2}U^{p-2}_i|\nabla_i U_i|^2_i\omega^n_i+\int_X2U_i^{p-2}e^{w_i}|\nabla_iw_i|_i^2(\Delta_i w_i)^2\omega^n_i\\
        &\leq \int_X\frac{1}{2}U^{p-2}_i|\nabla_i U_i|^2_i\omega^n_i+\int_X2U_i^{p-1}e^{\frac{1}{2}w_i}(\Delta_i w_i)^2\omega^n_i.\\
    \end{split}
\end{equation*}
In the second inequality we used $e^{\frac{1}{2}w_i}|\nabla_i w_i|^2_i\leq U_i$ again.

Putting these together and using $U_i\geq 1$, one can derive that (as in \cite[(3.54)]{CC2})
\begin{equation}
    \label{eq:int-U-p<int-U-p-G}
    \int_X\frac{p-1}{2}U_i^{p-2}|\nabla_i U_i|_i^2\omega_i^n\leq \int_XpU_i^p G_i e^{F_i}\omega^n,
\end{equation}
where 
$$
G_i:=\tilde G_i+(\Delta_i w_i)^2+2e^{\frac{1}{2}w_i}(\Delta_i w_i)^2+1.
$$

The rest of the proof uses Nash--Moser iteration, which goes through in exactly the same way as in \cite[p.960-962]{CC2}. Compared to \cite[(3.54)]{CC2}, in our $U_i$, there is an additional term $(n+\Delta u_{i})$, and in our $G_i$ there is an additional term $(n+\Delta u_{i-1})$. These additional terms will cause no trouble, thanks to \eqref{eq:Lp-bound-for-n+Delta-ui}.

For the reader's convenience, let us sketch this iteration process.

First, using H\"older's inequality as in \cite[(3.55)-(3.58)]{CC2}, we deduce from \eqref{eq:int-U-p<int-U-p-G} that
$$
||\nabla(U_i^{\frac{p}{2}})||^2_{L^{2-\varepsilon}(\omega^n)}\leq \frac{K_\varepsilon C p^3}{2(p-1)}\int_XU_i^p G_i e^{F_i}\omega^n,
$$
where (as in \cite[(3.57)]{CC2})
$$
K_\varepsilon=\left(\int_X(n+\Delta u_i)^{\frac{2}{\varepsilon}-1}\omega^n\right)^{\frac{\varepsilon}{2-\varepsilon}}
$$
and
$\varepsilon>0$ is some constant to be determined.

Then, applying the Sobolev inequality with exponent $(2-\varepsilon)$, we obtain (cf. \cite[(3.59)]{CC2})
$$
||U_i^{\frac{p}{2}}||^2_{L^{\frac{2n(2-\varepsilon)}{2n-2+\varepsilon}}}\leq D_\varepsilon\left(\frac{K_\varepsilon C p^3}{2(p-1)}\left(\int_XU_i^{\frac{2p}{2-\varepsilon}}\omega^n\right)^{\frac{2-\varepsilon}{2}}\times\left(\int_XG_i^{\frac{2}{\varepsilon}}e^{\frac{2F_i}{\varepsilon}}\omega^n\right)^{\frac{\varepsilon}{2}}+||U_i^{\frac{p}{2}}||^2_{L^{\frac{4}{2-\varepsilon}}}\right)
$$
where $D_\varepsilon>0$ depends on the Sobolev constant and $\int_X\omega^n$. Denote
$$
L_\varepsilon=\left(\int_XG_i^{\frac{2}{\varepsilon}}e^{\frac{2F_i}{\varepsilon}}\omega^n\right)^{\frac{\varepsilon}{2}}
$$
and choose $\varepsilon=\frac{1}{2n}$ so that
$$
\frac{2n(2-\varepsilon)}{2n-2+\varepsilon}>\frac{4}{2-\varepsilon}.
$$
Then we have that (cf. \cite[(3.62)]{CC2})
\begin{equation}
    \label{eq:iterate}
    ||U_i^{\frac{p}{2}}||^2_{L^{\frac{2n(2-\varepsilon)}{2n-2+\varepsilon}}}\leq \frac{Cp^3}{2(p-1)}(K_\varepsilon L_\varepsilon+1)||U_i^{\frac{p}{2}}||^2_{L^{\frac{4}{2-\varepsilon}}}.
\end{equation}

To carry out the iteration, we need to bound $K_\varepsilon$ and $L_\varepsilon$ (with $\varepsilon=\frac{1}{2n}$). Namely, we need to bound $\int_X(n+\Delta u_i)^{4n-1}\omega^n$ and $\int_X G_i^{4n}e^{4n F_i}\omega^n$. The former is bounded, thanks to \eqref{eq:Lp-bound-for-n+Delta-ui}. To bound $\int_X G_i^{4n}e^{4n F_i}\omega^n$, we first estimate $G_i$ as in \cite[(3.63)]{CC2}, which shows that
$$
G_i\leq C ((n+\Delta u_i)^{2n-1}+(n+\Delta u_{i-1})).
$$
So H\"older's inequality implies that
$$
L_\varepsilon\leq C\left(\int_X((n+\Delta u_i)^{2n-1}+(n+\Delta u_{i-1}))^{8n}\omega^n\right)^{\frac{1}{2}}\times\left(\int_Xe^{8nF_i}\omega^n\right)^{\frac{1}{2}},
$$
which is bounded due to \eqref{eq:Lp-bound-for-n+Delta-ui} and Proposition \ref{prop:C0}.

With these preparations, we can now iterate \eqref{eq:iterate} to get
$$
||U_i||_{L^\infty}\leq C ||U_i||_{L^1(\omega^n)}.
$$
It remains to show that $||U_i||_{L^1(\omega^n)}$ is bounded. Recall that
$$
U_i:=e^{\frac{1}{2}w_i}|\nabla_i w_i|^2_i+(n+\Delta u_i)+1.
$$
The first term has $L^1$-bound, as shown in the end of the proof of \cite[Theorem 3.2]{CC2}. The second term also has $L^1$-bound, thanks to \eqref{eq:Lp-bound-for-n+Delta-ui}. This completes the proof. 
\end{proof}

\begin{corollary}
    Assume that there is some constant $A>0$ such that
    $$
    Ent_\omega(u_i)+d_1(0,u_i)\leq A\text{ for all }i\in\NN.
    $$
    Then there exists some constant $B_4>1$ depending only on $X,\omega$ and $A$ such that
    $$
  B_4^{-1}\omega\leq\omega_i\leq B_4\omega\text{ for all }i\in\NN.
    $$
\end{corollary}

\begin{proof}
    This follows immediately from $\omega_i^n\geq e^{-B_1}\omega^n$ and $\tr_\omega\omega_i\leq B_3$.
\end{proof}

By classical elliptic estimates and bootstrapping, we then have the following uniform estimates.

\begin{corollary}
\label{cor:C-k-alpha-bound-for-ui}
    Assume that there is some constant $A>0$ such that
    $$
    Ent_\omega(u_i)+d_1(0,u_i)\leq A\text{ for all }i\in\NN.
    $$
    Then for any $\alpha\in(0,1)$ and $k\geq1$, there exists some constant $B_{k,\alpha}>1$ depending only on $X,\omega,\alpha,k$ and $A$ such that
    $$
  ||u_i||_{C^{k,\alpha}}\leq B_{k,\alpha}\text{ for all }i\in\NN.
    $$
\end{corollary}

\begin{proof}
    First, 
    the estimate $B_4^{-1}\omega\leq\omega_i\leq B_4\omega$ implies that \eqref{eq:laplace-of-Fi} is uniformly elliptic with bounded right hand side. Then arguing
    as in the proof of \cite[Proposition 4.2]{CC1}, one has $u_i\in C^{3,\alpha}$ and $F_i\in C^{1,\alpha}$ for all $i\in\NN$. This implies that the equation \eqref{eq:laplace-of-Fi} has $C^{1,\alpha}$-coefficients and right hand (since we already know that $u_{i-1}$ has $C^{3,\alpha}$ bound). This gives $C^{3,\alpha}$ bound for $F_i$. Differentiating the equation \eqref{eq:def-F-i-by-vol-form} twice one then gets a linear elliptic equation for the second derivatives of $u_i$ with $C^\alpha$ coefficients and right hand side. So we get the $C^{4,\alpha}$-bound for $u_i$. Continuing in this way we get all the $C^{k,\alpha}$ bounds for $u_i$.
\end{proof}

\begin{remark}
    In the above discussion we have set $\tau=1$ to simplify the exposition. In general, the equation we are dealing with is
    \begin{equation*}
        \begin{cases}
        (\omega+\ddc u_i)^n=e^{F_i}\omega^n,\\
            \Delta_{\omega_i}(F_i+u_{i-1}/\tau)=\tr_{\omega_i}(\Ric(\omega)-\omega/\tau)+n/\tau-\bar R.\\
        \end{cases}
    \end{equation*}
    In this case the estimates we get will depend on $\tau$ as well, and unfortunately they blow-up as $\tau\searrow 0$. Therefore, to show that the iteration converges to the flow \eqref{eqn:hKRF} as $\tau\searrow 0$, more effective estimates are needed.

    On the other hand, if we are allowed to take $\tau\gg0$ (for instance when the K-energy is bounded from below), then the boundedness assumption on $d_1(0,u_i)$ can be removed, since it is merely used to get the estimate \eqref{eq:skoda-estimate-for-u-i}:
    $$
    \int_Xe^{-p u_i/\tau}\omega^n\leq C,
    $$
    which now holds for free when $\tau\gg0$ by using Tian's $\alpha$-invariant \cite{Tian87} (here we need to normalize $u_i$ such that $\sup_X u_i=0$); see \cite[Lemma 4.20]{CC2} for a similar situation.
\end{remark}

\section{Smooth convergence of the Ricci iteration}
\label{sec:smooth-conv}

Assume that $(X,\omega)$ admits a cscK metric $\omega^*$ in $\{\omega\}$. Then by \cite[Theorem 1.5]{BDL20} we know that the K-energy is proper modulo $G:=Aut_0(X)$. Hence one can choose $\tau_0=\infty$ in Theorem \ref{thm:ite-exist}. Then for any $\tau>0$, we wish to show that the iteration sequence $\{\omega_i\}_{i\in\NN
}$ defined by \eqref{eqn:def-har-Ric-ite} converges in a suitable sense to a cscK metric. Up to scaling the K\"ahler class, we will assume without loss of generality that $\tau=1$. To make further simplification, we will first deal with the case where the cscK metric is unique, in which case the K-energy is proper (by \cite{DR17,BerBern17,BDL20}), i.e., $\gamma(X,\{\omega\})>0$ (recall \eqref{eq:def-gamma}). 

Therefore, we have that
\begin{itemize}
    \item There exists $\gamma>0$ and $C_0>0$ such that
    $$
    K_\omega(\varphi)\geq\gamma(I_\omega-J_\omega)(\varphi)-C_0\text{ for any }\varphi\in\cH_\omega.
    $$
    \item There exists a sequence $\{\omega_i\}_{i\in\NN}$ satisfying
    $$
    \omega_{i+1}-\omega_i=-\Ric(\omega_{i+1})+\HRic(\omega_{i+1}),\ \omega_0=\omega.
    $$
    Equivalently, one has
    \begin{equation}
        \label{eq:R-omega-i=R-n+tr}
        R(\omega_{i+1})=\bar R-n+\tr_{\omega_{i+1}}\omega_i,\ \omega_0=\omega.
    \end{equation}
    \item Write $\omega^*=\omega+\ddc{u^*}$ and $\omega_i=\omega+\ddc u_i$ for $u^*,u_i\in\cH_0$, where
    $$
    \cH_0:=\{\varphi\in\cH_\omega|E_\omega(\varphi)=0\}.
    $$
\end{itemize}

We wish to show the following.

\begin{theorem}
\label{thm:ui-d1-converge-to-cscK}
Assume that there exists a unique cscK potential $u^*\in\cH_0$, then the sequence $\{u_i\}_{i\in\NN}$ converges smoothly to $u^*$.    
\end{theorem}

To prove this, we need some preparations.

\begin{lemma}
One has
\begin{enumerate}
    \item $\omega_i$ minimizes $\cJ_\omega^{\omega_i}$ over $\cH_\omega$.
    \item  $\omega_{i+1}$ minimizes
    $K_\omega+\cJ^{\omega_i}_\omega$ over $\cH_\omega$.
\end{enumerate}
\end{lemma}

\begin{proof}
The first assertion follows from Lemma \ref{lem:cJ-cJ=J}.
    See e.g. \cite[Corollary 4.5]{CC2} for a proof of the second one. 
\end{proof}

\begin{lemma}
\label{lem:ent+d1-bound}
    One can find $A>0$ such that for all $i\in\NN$
    \begin{equation}
        \label{eq:bound-d1-Ent-for-ui}
        Ent_\omega(u_i)+d_1(0,u_i)\leq A.
    \end{equation}
\end{lemma}

\begin{proof}
    By Theorem \ref{thm:ite-decrease-K-energy} we have that (recall $\omega=\omega_0$)
    $$
    0=K_\omega(\omega_0)\geq K_\omega(\omega_i)=Ent_\omega(u_i)+\cJ_\omega^{-\Ric(\omega)}(u_i)\geq\gamma(I_\omega-J_\omega)(u_i)-C_0.
    $$
    This implies that (using Lemma \ref{lem:J<I-J<J})
    $$
    J_\omega(u_i)\leq n(I_\omega-J_\omega)(u_i)\leq \frac{n C_0}{\gamma}.
    $$
    So Lemma \ref{lem:J-d-1-compare} gives that
    $$
    d_1(0,u_i)\leq C_1.
    $$

    On the other hand, by \cite[Lemma 4.4]{CC2}, one has
    $$
    0\geq K_\omega(\omega_i)=Ent_\omega(u_i)+\cJ_\omega(u_i)\geq Ent_\omega(u_i)-C_2d_1(0,u_i).
    $$
    So we obtain that
    $$
    Ent_\omega(u_i)\leq C_3,
    $$
    finishing the proof.
\end{proof}

\begin{lemma}
\label{lem:J-ui+1-ui-small}
    One has for all $i\in\NN$
    $$
    J(u_{i+1},u_{i})\leq K_\omega(u_i)-K_\omega(u_{i+1}).
    $$
    So one has $I(u_{i+1},u_i)\to 0$.
\end{lemma}

\begin{proof}
    Using that $\omega_{i+1}$ minimizes $K_\omega+\cJ^{\omega_{i}}_\omega$, we have
    $$    K_\omega(u_{i+1})+\cJ^{\omega_{i}}_\omega(u_{i+1})\leq K_\omega(u_i)+\cJ^{\omega_{i}}_\omega(u_i).
    $$
    Then using $\cJ^{\omega_i}_\omega(u_{i+1})-\cJ_\omega^{\omega_i}(u_i)=J(u_{i+1},u_{i})$ (recall Lemma \ref{lem:cJ-cJ=J}) we conclude the first assertion.

    For the second statement, note that the $K$-energy is bounded from below in our setting, so Theorem \ref{thm:ite-decrease-K-energy} implies that $\{K_\omega(u_i)\}_{i\in\NN}$ is a convergent sequence. So we conclude from Lemma \ref{lem:J<I-J<J}.
\end{proof}

\begin{corollary}
\label{cor:u-ik-1-u-ik-same-d1-lim}
    If $\{u_{i_k}\}_{k\in\NN}$ is a $d_1$-convergent subsequence, say $u_{i_k}\xrightarrow{d_1}u$, then $u_{i_k-1}\xrightarrow{d_1}u$ as well. (Here $u_{i_k-1}$ denotes the $(i_k-1)$-th term in the sequence $\{u_i\}_{i\in\mathbb{N}}$)
\end{corollary}

\begin{proof}
    This follows from the previous lemma and Corollary \ref{cor:ui-vi-same-d1-lim}.
\end{proof}

Now we are ready to prove Theorem \ref{thm:ui-d1-converge-to-cscK}

\begin{proof}[Proof of Theorem \ref{thm:ui-d1-converge-to-cscK}]

We first argue that any convergent subsequence $\{u_{i_k}\}_{k\in\NN}$ has to converge to $u^*$ in the $d_1$-topology. By \eqref{eq:bound-d1-Ent-for-ui} and Lemma \ref{lem:d1-ent-bound-cpt} this will imply that $u_i\xrightarrow{d_1}u^*$.

So assume that there exists a subsequence $\{u_{i_k}\}_{k\in\NN}$, converging in $d_1$ to a limit $u_\infty\in\cE^1_\omega$.  Then for any $u\in\cH_\omega$ we deduce that
\begin{equation*}
    \begin{aligned}        K_\omega(u_\infty)&\leq\lim_{k\to\infty} K_\omega(u_{i_k})\\
        &=\lim_{k\to\infty}(K_\omega(u_{i_k})+\cJ_\omega^{\omega_{i_k-1}}(u_{i_k})-\cJ_\omega^{\omega_{i_k-1}}(u_{i_k}))\\
        &\leq\lim_{k\to\infty}(K_\omega(u)+\cJ_\omega^{\omega_{i_k-1}}(u)-\cJ_\omega^{\omega_{i_k-1}}(u_{i_k-1}))\\
        &=\lim_{k\to\infty}(K_\omega(u)+J(u,u_{i_k-1}))=K_\omega(u)+J(u,u_\infty).\\
    \end{aligned}
\end{equation*}
Here we used that $K_\omega$ is $d_1$-lsc, $u_{i_k}$ minimizes $K_\omega+\cJ_\omega^{\omega_{i_k-1}}$, $u_{i_k-1}$ minimizes $\cJ_\omega^{\omega_{i_k-1}}$, Lemma \ref{lem:cJ-cJ=J}, Corollary \ref{cor:u-ik-1-u-ik-same-d1-lim} and Lemma \ref{lem:J-d1-conti}.
Thus we obtain that
$$
K_\omega(u_\infty)\leq K_\omega(u)+J(u,u_\infty)\ \text{for any }u\in\cH_\omega.
$$
By Lemma \ref{lem:funtionals-to-E1} we then see that $u_\infty$ is a minimizer of the functional
$$
F_\infty(u):=K_\omega(u)+J(u,u_\infty), u\in\cE^1_\omega.
$$

We now argue that $u_\infty$ must be a cscK potential and hence $u_\infty=u^*$.

By Lemma \ref{lem:ent+d1-bound}, Corollary \ref{cor:C-k-alpha-bound-for-ui} and Arzel\`a--Ascoli, we know that $u_\infty\in\cH_\omega$. So by Lemma \ref{lem:cJ-cJ=J} we can write
$$
F_\infty(u)=K_\omega(u)+\cJ_\omega^{\omega_{u_\infty}}(u)-\cJ^{\omega_{u_\infty}}_\omega(u_\infty)=K_\omega^{\omega_{u_\infty}}(u)-\cJ^{\omega_{u_\infty}}_\omega(u_\infty).
$$
So $u_\infty$ minimizes the twisted K-energy $K_\omega^{\omega_{u_\infty}}$.
The variation formula \eqref{eq:var-formula} of $K_\omega^{\omega_{u_\infty}}$ then implies that
$$
R(\omega_{u_\infty})=\bar R-n+\tr_{\omega_{u_\infty
}}\omega_{u_\infty}=\bar R.
$$
Thus $\omega_{u_\infty}$ is a cscK metric. By uniqueness assumption we have that $u_\infty=u^*$.

Therefore, we have shown that $u_{i_k}\xrightarrow{d_1}u^*$ for any convergent subsequence. So $u_i\xrightarrow{d_1}u^*$ follows. By Corollary \ref{cor:C-k-alpha-bound-for-ui} and Arzel\`a--Ascoli we then know that $u_i\to u^*$ smoothly.
\end{proof}

If we do not assume the uniqueness of the cscK  metric $\omega^*$, then the K-energy is proper modulo the action of biholomorphic automorphisms of $X$ (see \cite[Theorem 1.5]{BDL20}). Modifying our previous proofs and incorporating the ideas from \cite{DR19}, one can actually prove the following result, which improves Theorem \ref{thm:ui-d1-converge-to-cscK} and extends Darvas--Rubinstein's work \cite[Theorem 1.6]{DR19} to arbitrary K\"ahler classes. 

\begin{theorem}(=Theorem \ref{thm:smooth-convergence})
\label{thm:smooth-convergence-modulo-actions}
    Let $(X,\omega)$ be a compact K\"ahler manifold admitting a cscK metric in $\{\omega\}$. Then for any $\tau>0$ the iteration sequence \eqref{eqn:def-har-Ric-ite} sequence exists and there exist holomorphic diffeomorphsims $g_i$ such that $g_i^*\omega_i$ converges smoothly to a cscK metric.
\end{theorem}

\begin{proof}
We give the necessary details for the reader's convenience. As above, we assume without loss of generality that $\tau=1$.

    First, using that the K-energy decreases along $u_i$ and is proper modulo $G=Aut_0(X)$, we have that (recall \eqref{eq:def-d-G})
    $$d_{1,G}(0,u_i)\leq A_0\text{ for all }i\in\NN.$$
    Fix a cscK metric $\omega^*\in\{\omega\}$ with $\omega^*=\omega+\ddc u^*$ and $u^*\in\cH_0$.
    Then pick $g_i\in G$ such that
    \begin{equation}
        \label{eq:def-g-i}
        d_1(u^*,g_i.u_i)\leq d_{1,G}(u^*,u_i)+\frac{1}{i}\leq d_{1}(0,u^*)+d_{1,G}(0,u_i)\leq A_1.
    \end{equation}
    Thus we deduce that
    $$
    d_1(0,g_i.u_i)\leq A_1+d_1(0,u^*).
    $$
    Then using that the K-energy is $G$-invariant (see e.g. \cite[Lemma 4.11]{CC2}), one can argue as in the proof of Lemma \ref{lem:ent+d1-bound} to show that
    $$
    Ent_\omega(g_i.u_i)\leq A_2\text{ for all }i\in\NN.
    $$
    So the sequence $\{g_i.u_i\}_{i\in\NN}$ is $d_1$-precompact. We wish to show that it converges to $u^*$ smoothly. To this end, we need some uniform estimates for the sequence.

    By Lemmas \ref{lem:I&J-G-inv}, \ref{lem:J<I-J<J}, \ref{lem:J-ui+1-ui-small} and Theorem \ref{thm:ite-decrease-K-energy}, we know that
    \begin{equation*}
        \begin{aligned}
             I(g_i.u_{i},g_i.u_{i-1})&=I(u_{i},u_{i-1})\leq (n+1)J(u_{i},u_{i-1})\\
             &\leq (n+1)(K_\omega(u_{i-1})-K_\omega(u_{i}))\to 0.\\
        \end{aligned}
    \end{equation*}
    And also, one has (by Lemma \ref{lem:triangle-ineq-for-I})
    $$
    J(0,g_i.u_{i-1})\leq I(0,g_i.u_{i-1})\leq C_n(I(0,g_i.u_i)+I(g_i.u_{i},g_i.u_{i-1})).
    $$
    So we derive that (using Lemma \ref{lem:J-d-1-compare})
    $$
    d_1(0,g_i.u_{i-1})\leq A_3.
    $$
The upshot is that, there exists some $A>0$ such that
$$
Ent_\omega(g_i.u_i)+d_1(0,g_i.u_{i-1})\leq A\text{ for all }i\geq 1.
$$ 

For simplicity let us put
$$
v_i:=g_i.u_i\text{ and }h_{i-1}:=g_i.u_{i-1}.
$$
Then from \eqref{eq:R-omega-i=R-n+tr} we deduce that
$$
R(\omega_{v_i})=\bar R-n+\tr_{\omega_{v_i}}(\omega+\ddc h_{i-1}).
$$
This is equivalent to  (cf. \cite[Lemma 4.19]{CC2})
\begin{equation}
\label{eq:coupled-eqn-hi-vi}
    \begin{cases}
        (\omega+\ddc v_i)^n=e^{F_i}\omega^n,\\
        \Delta_{\omega_{v_i}}(F_i+h_{i-1})=\tr_{\omega_{v_i}}(\Ric(\omega)-\omega)+n-\bar R.\\
    \end{cases}
\end{equation}
And we have that
$$
Ent_\omega(v_i)+d_1(0,h_{i-1})\leq A\text{ for all }i\geq 1.
$$
Then as in Proposition \ref{prop:C0} we can obtain the $C^0$ estimate:
$$
|v_i|\leq B_1\text{ for all }i\geq1.
$$
This implies that (by Lemma \ref{lem:d1-comparable-L1})
$$
d_1(0,v_i)=d_1(0,g_{i}.u_i)=d_1(g_{i}^{-1}.0,u_i)=d_1(g_{i+1}.(g_{i}^{-1}.0),h_{i+1})\leq B_2.
$$
Put
$$
f_i:=g^{-1}_{i}\circ g_{i+1}.
$$
Then we have
$$
d_1(f_i.0,0)\leq d_1(f_i.0,h_{i+1})+d_1(0,h_{i+1})\leq B_2+A_3\text{ for all }i\geq 1.
$$
By the proof of \cite[Proposition 6.8]{DR17}, $\{f_i\}_{i\geq 1}$ is contained in a bounded set of $G.$ In particular, all derivatives of $f_i$ up to order $m$, say, are bounded by some $C_m$ independently of $i$. Since one has
$$
h_i=g_{i+1}.u_i=f_i.v_i,
$$
then $v_i$ and $h_i$ enjoy the same a priori estimates. Now the same arguments as in \S \ref{sec:estimates} apply to the system of equations \eqref{eq:coupled-eqn-hi-vi} as well. We conclude that there are uniform $C^{k,\alpha}$ estimates (independent of $i$) for $v_i$ and $h_i$.

Now we are ready to show that $v_i\to u^*$ smoothly. 

By Arzel\`a--Ascoli it suffices to argue that $v_i\xrightarrow{d_1}u^*$. We prove by contradiction. Assume that there exists a subsequence such that $v_{i_k}\xrightarrow{d_1}v_\infty$ for some $v_\infty\in\cE^1_\omega$ with $d_1(u^*,v_\infty)>\varepsilon>0$. By our uniform estimates for $v_i$ and Arzel\`a--Ascoli we know that $v_\infty\in\cH_0.$

For any $u\in\cH_0$, one has (as in the proof of Theorem \ref{thm:ui-d1-converge-to-cscK})
\begin{equation*}
    \begin{aligned}
    K_\omega(v_\infty)&\leq\lim_{k\to\infty}K_\omega(v_{i_k})=\lim_{k\to\infty}K_\omega(u_{i_k})\\
        &=\lim_{k\to\infty}(K_\omega(u_{i_k})+\cJ^{\omega_{u_{i_k-1}}}(u_{i_k})-\cJ^{\omega_{u_{i_k-1}}}(u_{i_k}))\\
        &\leq\lim_{k\to\infty}(K_\omega(g_{i_k}^{-1}.u)+\cJ^{\omega_{u_{i_k-1}}}(g_{i_k}^{-1}.u)-\cJ^{\omega_{u_{i_k-1}}}(u_{i_k-1}))\\
        &=\lim_{k\to\infty}(K_\omega(u)+J(g^{-1}_{i_k}.u,u_{i_k-1}))\\
        &=\lim_{k\to\infty}(K_\omega(u)+J(u,h_{i_k-1}))=K_\omega(u)+J(u,v_\infty).\\
    \end{aligned}
\end{equation*}
Here we used that $K_\omega$ and $J$ are $G$-invariant (recall Lemma \ref{lem:I&J-G-inv}).
Moreover, in the last equality we used that $h_{i_k-1}\xrightarrow{d_1}v_\infty$. Indeed, Lemma \ref{lem:I&J-G-inv} and \ref{lem:J-ui+1-ui-small} imply that $I(v_i,h_{i-1})=I(u_{i},u_{i-1})\to 0$. So 
Corollary \ref{cor:ui-vi-same-d1-lim} implies that $\lim_{k}h_{i_k-1}=\lim_k v_{i_k}=v_\infty$, as claimed.

From above we observe that $v_\infty$ is a minimizer for the functional
$$
F_\infty(u):=K_\omega(u)+J(u,v_\infty),\ u\in\cH_0.
$$
This further implies that $v_\infty$ is a minimizer of $F_\infty$ over $\cH_\omega$. Then as in the proof of Theorem \ref{thm:ui-d1-converge-to-cscK}, we conclude that $v_\infty$ is a cscK potential. 

By \cite[Theorem 1.3]{BerBern17} there exists $f\in G$ such that $v_\infty=f.u^*$. So we obtain that (recall \eqref{eq:def-g-i} and that $v_{i_k}=g_{i_k}.u_{i_k}$)
$$
d_1(v_{i_k},u^*)-\frac{1}{i_k}\leq d_{1,G}(v_{i_k},u^*)\leq d_1(f^{-1}.v_{i_k},u^*)=d_1(v_{i_k},v_\infty).
$$
By choice the right hand goes to zero, while the left hand side is strictly bigger than $\frac{\varepsilon}{2}>0$ for any $k\gg1$. This is a contradiction. So we finish the proof.
\end{proof}

As in \cite{DR19}, one expects that the appearance of $g_i$ in the above theorem is actually redundant, which might require substantial new ideas; compare also \cite[Proposition 4.17]{CC2} in the setting of continuity method. In the case of K\"ahler Ricci flow, the analogous problem is studied in \cite{TZ07,PSSW08,TZ13,TZZZ13,CoSz16}.

\section{The case of twisted cscK metrics}
\label{sec:thR-ite}

In the study of cscK metrics, it is often beneficial to allow for some twisted terms. More precisely, given a closed smooth real $(1,1)$ form, one can study the following $\chi$-twisted cscK equation:
\begin{equation}
    \label{eq:def-tcscK}
    R(\omega_u)=\bar R-\bar \chi+\tr_{\omega_u}\chi.
\end{equation}
This is equivalent to saying that $\Ric(\omega_u)-\chi$ is harmonic with respect to $\omega_u$.
Therefore, to search for $\chi$-twisted cscK metrics, we are led to the following twisted flow:
\begin{equation}
    \label{eq:def-thKRF}
    \partial_t\omega_t=-\Ric(\omega_t)+\mathrm H_{\omega_t}(\Ric(\omega_t)-\chi)+\chi,\ \omega_0=\omega.
\end{equation}
Here $H_{\omega_t}$ denotes the harmonic projection operator of the metric $\omega_{t}$.
When $2\pi c_1(X)=\lambda\{\omega\}+\{\chi\}$, this flow becomes
$$
\partial_t\omega_t=-\Ric(\omega_t)+\lambda\omega_t+\chi,\ \omega_0=\omega,
$$
which is the twisted K\"ahler Ricci flow studied in \cite{Liu13,CoSz16}.

Discretizing the flow \eqref{eq:def-thKRF}, we get (for some given $\tau>0$)
\begin{equation}
    \label{eq:def-thR-ite}
    \frac{\omega_{i+1}-\omega_i}{\tau}=-\Ric(\omega_{i+1})+\mathrm H_{\omega_{i+1}}(\Ric(\omega_{i+1})-\chi)+\chi,\ i\in\NN,\ \omega_0=\omega.
\end{equation}

The next result can be proved following exactly the same strategy as we did for the untwisted case. Hence we omit the details.

\begin{theorem}
\label{thm:thRic-ite-converge}
Assume that $\chi\geq0$.
    There exists a constant $\tau_0\in(0,\infty]$ depending only on $X,\{\omega\}$ and $\{\chi\}$ such that for any $\tau\in(0,\tau_0)$, the iteration sequence \eqref{eq:def-thR-ite} exists for all $i\in\NN$, with each $\omega_i$ being uniquely determined by $\omega_0$, along which the $\chi$-twisted K-energy $K_\omega^\chi$ decreases. Moreover, if there exists a unique $\chi$-twisted cscK metric $\omega^*\in\{\omega\}$, then for any $\tau>0$ the sequence $\omega_i$ converges to $\omega^*$ smoothly.
\end{theorem}

Note that, if $\chi>0$, the uniqueness of $\chi$-twisted cscK metric is automatic by \cite[Theorem 4.13]{BDL17}. This might be useful, since one can study the flow \eqref{eqn:hKRF} or the iteration \eqref{eqn:def-har-Ric-ite} by adding a small amount of $\chi$ (cf. the perturbation trick in \cite[\S 4]{BerBern17}).

One can try to extend our work further to the case of conical cscK metrics, extremal metrics and other canonical metrics. We leave this to the interested readers.

\AtNextBibliography{\small}
\begingroup
\setlength\bibitemsep{2pt}
\printbibliography

\end{document}